\documentclass[10pt, letterpaper,reqno]{amsart}
\usepackage{amssymb}
\usepackage{amsmath}
\usepackage{amsfonts, color}
\usepackage{graphicx}
\usepackage{mathrsfs,amssymb, slashed, cite}

\usepackage{hyperref}
\usepackage{cleveref}

\newcommand{\R}{\mathbb{R}}
\newcommand{\C}{\mathbb{C}}

\newcommand{\F}{\mathcal{F}}

\newtheorem{lemma}{Lemma}[section]

\newtheorem{proposition}{Proposition}[section]

\newtheorem{theorem}{Theorem}[section]
\numberwithin{equation}{section}

\newcommand{\eps}{\varepsilon}

\newcommand{\qtq}[1]{\quad\text{#1}\quad}

\begin{document}

\title[Dispersion-managed NLS]{Modified scattering for the cubic dispersion-managed NLS}

\author{Jason Murphy}
\address{Department of Mathematics, University of Oregon, Eugene, OR, USA.}
\email{jamu@uoregon.edu}

\author{Jiqiang Zheng}
\address{Institute of Applied Physics and Computational Mathematics and National Key Laboratory of Computational Physics, Beijing 100088, China.}
\email{zheng\_jiqiang@iapcm.ac.cn}

\maketitle 

\begin{abstract} We establish a small-data modified scattering result for the $1d$ cubic dispersion-managed NLS (with time-dependent dispersion map) for initial data in a weighted space. 
\end{abstract}

\section{Introduction}

We consider the cubic dispersion-managed nonlinear Schr\"odinger equation in one space dimension:
\begin{equation}\label{dmnls}
i\partial_t u + \gamma(t)\Delta u = - |u|^2 u,\quad u:\R_t\times\R_x^d\to\C.
\end{equation}
Here the dispersion map $\gamma$ is a 1-periodic function of time. We will focus on the physically-relevant case of a piecewise-constant dispersion map, with positive average dispersion.  In particular, we consider the 1-periodic extension of 
\begin{equation}\label{gamma1}
\gamma(t) = \begin{cases} \gamma_+ & t\in[0,\tfrac12), \\ -\gamma_- & t\in[\tfrac12,1)\end{cases}
\end{equation}
for some $\gamma_\pm>0$, with 
\begin{equation}\label{gamma2}
\langle \gamma\rangle = \int_0^1 \gamma(t)\,dt = \tfrac12[\gamma_+-\gamma_-]>0. 
\end{equation}

We remark that the analysis in this paper extends naturally to more general dispersion maps with nonzero average dispersion.  For example, the analysis applies to the class of dispersion maps treated in \cite{MVH}, which includes $1$-periodic $\gamma$ such that (i) $\langle\gamma\rangle\neq 0$, (ii) $\gamma$ and $\gamma^{-1}$ are bounded, and (iii) $\gamma$ has at most finitely many discontinuities on $[0,1]$. 

We choose the nonlinear coefficient to be negative, corresponding to a self-focusing nonlinearity.  Of course, in the small-data regime considered in this work, the sign of the nonlinearity is essentially irrelevant. 

The equation \eqref{dmnls} models the propagation of short laser pulses through fiber optic cables in the presence of dispersion management (see e.g. \cite{Hasegawa, Mitschke, TuritsynBF}).  In this setting, two different types of cable are alternated periodically, with one type corresponding to normal dispersion at the carrier frequency and the other anomalous dispersion.  

Dispersion-managed nonlinear Schr\"odinger equations (DMNLS) have received a great deal of recent interest in both the optics and mathematics communities.  For a  sample of recent works, we refer the reader to \cite{EHL, Pel, PZ, AntonelliSautSparber, CMVH, ChoiLee, ZGJT, MVH, KawakamiMurphy, CHL, HuLe, MVH0, GT1} and the references therein. 

We establish decay and modified scattering for small initial data in the weighted Sobolev space $\Sigma=\{f\in H^1: xf\in L^2\}$.  Our result parallels the well-result for the standard cubic NLS (see \cite{HN, IT, KP, LS, DZ, MurphyReview}), as well as the result for the Gabitov--Turitsyn dispersion-managed NLS \cite{MVH0}, which is essentially an averaged version of \eqref{dmnls} (cf. \cite{CHL, GT1}).  The novelty in the present work is to deal directly with the time-dependent dispersion map $\gamma(t)$. 

Our result may be stated precisely as follows:

\begin{theorem}\label{T1} Let $d=1$ and define $\gamma$ as in \eqref{gamma1}--\eqref{gamma2}. Let $u_0\in\Sigma$ satisfy $\|u_0\|_{\Sigma}=\eps>0$, and let $u:\R\times[0,\infty)\to\C$ be the unique solution to \eqref{dmnls} with $u|_{t=0}=u_0$.

If $\eps$ is sufficiently small, then $u$ satisfies the decay estimate
\[
\|u(t)\|_{L^\infty(\R)}\lesssim \eps \langle t\rangle^{-\frac12}\qtq{uniformly in}t\geq 0. 
\]
Furthermore, there exists $W\in L^\infty$ such that 
\begin{equation}\label{asymptotic}
\lim_{t\to\infty} \| u(t,x) - (2i\Gamma(t))^{-\frac12} \exp\bigl\{\tfrac{ix^2}{4\Gamma(t)}+\tfrac{i}{2\langle\gamma\rangle}|W(\tfrac{x}{2\Gamma(t)})|^2\log t\bigr\}W(\tfrac{x}{2\Gamma(t)})\|_{L_x^\infty} = 0,
\end{equation}
where $\Gamma(t):=\int_0^t \gamma(s)\,ds$ is the total dispersion.
\end{theorem}

Theorem~\ref{T1} describes the long-time behavior of solutions as $t\to\infty$.  Repeating the arguments yields a similar result for $t\to-\infty$.  We remark that the global existence of solutions is relatively straightforward to obtain (see Proposition~\ref{P:GWP1d}).  In particular, the decay and asymptotic behavior are the main points of the theorem. 

To prove Theorem~\ref{T1}, we will make use of the vector field
\[
J_\Gamma(t,t_0) = x+2i\Gamma(t,t_0)\nabla,\quad \Gamma(t,t_0):=\int_{t_0}^t \gamma(s)\,ds,
\]
which is the generalization of the Galilean vector field $J(t)=x+2it\nabla$ used in the scattering theory for the standard NLS (see e.g. \cite{Cazenave} for a textbook treatment or \cite{MurphyReview0, MurphyReview} for some reviews).  In fact, $J(t)$ has been used effectively in the scattering theory for the Gabitov--Turitsyn dispersion-managed NLS, as well (see e.g. \cite{KawakamiMurphy, MVH0}).  This reflects the fact that these models have a Galilean symmetry, and it is for the same reason that this vector field can be used effectively in the present setting. 

To prove Theorem~\ref{T1}, we will adapt some of the arguments used to establish small-data modified scattering for the standard NLS.  In particular we will rely on a factorization of the propagator $e^{i\Gamma(t)\Delta}$ in the spirit of \cite{HN}, as well as a change of variables in the spirit of \cite{LS} (see also \cite[Section~5]{MurphyReview}).  As in all of the proofs of modified scattering for cubic NLS, the core of the proof is a bootstrap argument relating an `energy' norm (involving $\|J_\Gamma(t)u(t)\|_{L^2}$) and a `dispersive' norm (involving $\|u(t)\|_{L^\infty}$).  We control the energy norm using a chain-rule estimate for $J_\Gamma$ and Gronwall's inequality, while the dispersive norm is estimated by introducing a new variable, denoted $w$ in this work, and employing an integrating factor in the equation for $w$.  After closing the estimates, the asymptotic behavior is established by further analyzing the equation for $w$. 

To close this introduction, let us compare the asymptotic formula appearing in \eqref{asymptotic} with those arising in the case of the standard cubic NLS and Gabitov--Turitsyn NLS.  We can then compare with the standard cubic NLS
\begin{equation}\label{standard}
i\partial_t u + \langle\gamma\rangle\Delta u = -|u|^2 u,
\end{equation}
as well as the Gabitov--Turitsyn NLS, which may be expressed in this setting as
\begin{equation}\label{GT}
i\partial_t u + \langle\gamma\rangle\Delta u = -\int_0^1 e^{-iD(\tau)\Delta}\{|e^{iD(\tau)\Delta}u|^2 e^{iD(\tau)\Delta}u\}\,d\tau, 
\end{equation}
with  $D(\tau):=\Gamma(\tau)-\langle\gamma\rangle\tau$. In both cases, the long-time behavior of solutions is given by
\begin{equation}\label{asymptotic2}
u(t,x) = (2i\langle\gamma\rangle t)^{-\frac12}\exp\{ \tfrac{ix^2}{4\langle\gamma\rangle t}+\tfrac{i}{2\langle\gamma\rangle} |W(\tfrac{x}{2\langle\gamma\rangle t})|^2\log t\} W(\tfrac{x}{2\langle\gamma\rangle t})+o(t^{-\frac12})
\end{equation}
as $t\to\infty$ for some profile $W$ (cf. \cite{HN, IT, KP, LS, DZ, MurphyReview, MVH0}). 

As a matter of fact, the long-time behavior in \eqref{asymptotic2} agrees exactly with that in \eqref{asymptotic}.  Indeed, if we were to denote the total dispersion $\langle \gamma\rangle t$ in \eqref{asymptotic2} by $\Gamma(t)$ (as we do for \eqref{dmnls}), the formulas would become identical. Put differently, \eqref{asymptotic2} is exactly what appears in \eqref{asymptotic} if we specialize to the case $\gamma(t)\equiv \langle\gamma\rangle$.

\subsection*{Acknowledgements} J. M. was supported by NSF DMS-2350225. 
J. Z. was supported by National key R\&D program of China: 2021YFA1002500 and NSF grant of China (No. 12271051, 12426502). We are grateful to Yanfang Gao, who hosted us at Fujian Normal University while part of this work was completed. 

\section{Preliminaries}

We use the standard $\lesssim$ notation, i.e. we write $A\lesssim B$ to denote $A\leq CB$ for some $C>0$.  We indicate dependence of the constant on parameters via subscripts, e.g. $A\lesssim_T B$ denotes $A\leq CB$ for some $C=C(T)$. If $A\lesssim B$ and $B\lesssim A$, we write $A\sim B$. 

We define the Fourier transform on $\R$ via
\[
\hat f(\xi) = (2\pi)^{-\frac{1}{2}}\int_{\R} e^{-ix\cdot\xi}f(x)\,dx. 
\]

We employ the standard Littlewood--Paley frequency projections, denoted by $P_N$, $P_{\leq N}$, $P_{>N}$.  Here $P_N$ is the Fourier multiplier operator $\F^{-1} \psi(\tfrac{\xi}{N}) \F$, with $\psi$ a smooth cutoff to frequencies $|\xi|\sim 1$.  Similarly, $P_{\leq N}$ restricts to frequencies $|\xi|\lesssim N$ and $P_{>N}$ restricts to frequencies $|\xi|\gtrsim N$.  These operators are bounded on all $L^p$ spaces and obey the following standard estimates, which may be proven via Young's convolution inequality:
\begin{lemma}[Bernstein estimates]\label{L:Bernstein} For $1\leq p\leq q\leq\infty$ and $s\in\R$:
\[
\|P_N f\|_{L^q(\R)} \lesssim N^{\frac{1}{p}-\frac{1}{q}}\|P_N f\|_{L^p(\R)},\quad \||\nabla|^s P_N f\|_{L^p(\R)}\sim_s N^s \|P_N f\|_{L^p(\R)}. 
\]
\end{lemma} 
Here $|\nabla|^s$ is the fractional derivative defined as the Fourier multiplier operator $\F^{-1} |\xi|^s \F$.

The free Schr\"odinger propagator $e^{it\Delta}$ may be defined as the Fourier multiplier operator $\F^{-1} e^{-it\xi^2} \F$.  We will make use of the following factorization identity, which follows from direct calculation:
\begin{equation}\label{factorization}
e^{it\Delta} = M(t) D(t) \F M(t),
\end{equation}
where $M(t)=e^{\frac{ix^2}{4t}}$ and $D(t)$ is the dilation operator
\[
[D(t)f](x) = (2it)^{-\frac{1}{2}}f(\tfrac{x}{2t}). 
\]

In particular (recalling $\Gamma(t)=\Gamma(t,0)=\int_0^t \gamma(s)\,ds$), we have the identity
\begin{equation}\label{factorization}
e^{i\Gamma(t)\Delta} = M(\Gamma(t)) D(\Gamma(t))\F M(\Gamma(t)),
\end{equation}
which we will only use on intervals on which $\Gamma(t)>0$. 

We note that for the class of dispersion maps under consideration in this paper, we have the following basic but useful estimate concerning the total dispersion (see \cite[Lemma~1]{MVH}):
\begin{equation}\label{Gammalemma}
|\Gamma(t) - t\langle\gamma\rangle| \leq 2\|\gamma\|_{L^\infty}
\end{equation}
for all $t\in\R$.  In particular, we have that $\Gamma(t)>0$ for all $t>\tfrac{2\|\gamma\|_{L^\infty}}{\langle\gamma\rangle}$.  In fact, we have 
\begin{equation}\label{Gammalemma2}
\Gamma(t)\geq \tfrac12\langle \gamma\rangle t\qtq{for all}t\geq T_0:=\tfrac{4\|\gamma\|_{L^\infty}}{\langle\gamma\rangle}.
\end{equation}

We next define the vector field $J_\Gamma(t,t_0)$ and observe two useful representations of this operator: 
\begin{equation}\label{J-identities}
\begin{aligned}
J_\Gamma(t,t_0) =x+2i\Gamma(t,t_0)\nabla &= e^{i\Gamma(t,t_0)\Delta} x e^{-i\Gamma(t,t_0)\Delta} \\
& = e^{\frac{i|x|^2}{4\Gamma(t,t_0)}}[2i\Gamma(t,t_0)\nabla] e^{-\frac{i|x|^2}{4\Gamma(t,t_0)}}. 
\end{aligned}
\end{equation}

The vector field $J_\Gamma$ obeys the following commutation property with the free propagator:
\[
J_\Gamma(t,t_0)e^{i\Gamma(t,s)\Delta} = e^{i\Gamma(t,s)\Delta}J_\Gamma(s,t_0),
\]
which is readily derived from the first representation on the right-hand side of \eqref{J-identities}. By direct calculation, we also have the following pointwise chain rule estimate:
\begin{equation}\label{chain}
|J_\Gamma(t,t_0)[|u|^2 u]| \lesssim |u|^2 |J_\Gamma(t,t_0)u|,
\end{equation}
which will be used several times in what follows. 

Just as we abbreviate $\Gamma(t,0)$ by $\Gamma(t)$, we will abbreviate $J_\Gamma(t,0)$ by $J_\Gamma(t)$.

\section{Modified scattering}\label{S:main}

In this section we prove the main theorem, Theorem~\ref{T1}. To begin, we obtain global existence for \eqref{dmnls} with data in $\Sigma$.

\begin{proposition}\label{P:GWP1d} For any $u_0\in L^2(\R)$ and $t_0\in\R$, there exists a unique global solution $u:\R\times\R\to\C$ to \eqref{dmnls} with $u|_{t=t_0}=u_0$.  In addition, if $u_0\in \Sigma$, then $u\in C(\R;\Sigma)$. 
\end{proposition}

\begin{proof} Let $u_0\in L^2$.  We will first construct a solution to the Duhamel formula
\[
u(t) = e^{i\Gamma(t,t_0)\Delta}u_0+i\int_{t_0}^t e^{i\Gamma(t,s)\Delta}|u|^2 u(s)\,ds. 
\]
on an interval $[t_0,t_0+T]$, where $T=T(\|u_0\|_{L^2})$. 
As in the case of the standard NLS, the key tool we need is the set of Strichartz estimates for the underlying linear equation.  For the particular type of dispersion map under consideration here, the work \cite{MVH} proved that the dispersion-managed equation admits essentially the same set of Strichartz estimates as the usual linear Schr\"odinger equation (only the double $L_t^2$-endpoint in dimensions $d\geq 3$ was not obtained).  In light of this fact, local well-posedness for \eqref{dmnls} in $L^2$ follows from the same arguments used to handle the standard NLS (see e.g. \cite{Cazenave} for a textbook treatment). 

Using the conservation of mass, we can iterate to obtain global existence in $L^2$.  It remains to show that if $u_0\in\Sigma$, then $u\in C(\R;\Sigma)$.  As spatial derivatives commute with the linear part of the equation, it is straightforward to check that $u_0\in H^1$ yields $u\in C(\R;H^1)$.  To prove that $u_0\in H^{0,1}$ leads to $u\in C(\R;H^{0,1})$, we use the vector field 
\[
J_\Gamma(t,t_0)=x+2i\Gamma(t,t_0)\nabla
\]
and observe that it suffices to estimate locally in time.  We focus on the inhomogeneous term in the Duhamel formula and use the commutation properties of $J_\Gamma$ and the chain rule estimate \eqref{chain} for $J_\Gamma$ to estimate on a short interval $[t_0,t_0+T]$:
\begin{align*}
\biggl\| & J_\Gamma(t,t_0)\int_{t_0}^t e^{i\Gamma(t,s)\Delta}|u|^2 u(s)\,ds\biggr\|_{L_t^\infty L_x^2} \\
& \lesssim \biggl\|\int_{t_0}^t e^{i\Gamma(t,s)\Delta}J_\Gamma(s,t_0)[ |u|^2 u(s)]\,ds\biggr\|_{L_t^\infty L_x^2} \\
& \lesssim T\|u\|_{L_{t,x}^\infty}^2 \|J_{\Gamma}(t,t_0)u\|_{L_t^\infty L_x^2} \\
& \lesssim T \|u\|_{L_t^\infty H_x^1}^2 \|J_{\Gamma}(t,t_0)u\|_{L_t^\infty L_x^2}.
\end{align*}
Estimating in this way, we may obtain $J_\Gamma(t,t_0)u\in C_t L_x^2$.  By the continuity in $H_x^1$ and the triangle inequality, this yields finally yields continuity of $xu$ in $L_x^2$. 
\end{proof}

We now let $u_0\in\Sigma$ with $\|u_0\|_{\Sigma}=\eps$, with $0<\eps\ll1$, and we take $u:[0,\infty)\times\R\to\C$ to be the corresponding solution to \eqref{dmnls} provided by Proposition~\ref{P:GWP1d}.  The proof of Theorem~\ref{T1} will rely on a bootstrap argument that controls two norms of the solution, which we denote by
\[
\|u(t)\|_X := \sup_{s\in[0,t]} \bigl\{\|u(s)\|_{L_x^2}+\langle s\rangle^{-\delta}\|J_\Gamma(s)u(s)\|_{L_x^2}+ \langle s\rangle^{-\delta}\|\nabla u(s)\|_{L_x^2}\}
\]
for some small $\delta>0$ to be specified below, and
\[
\|u(t)\|_S := \sup_{s\in[0,t]}\langle s\rangle^{\frac12}\|u(s)\|_{L_x^\infty}.
\]
The main decay estimate appearing in the proof of Theorem~\ref{T1} is equivalent to the statement that $\|u(t)\|_S\lesssim \eps$ uniformly in $t\geq 0$. 

Using conservation of mass, the $L^2$-component of the $X$-norm is \emph{a priori} controlled by $\eps$.\footnote{In contrast to the standard NLS, there is no globally conserved energy for \eqref{dmnls} that we can use to obtain \emph{a priori} $\dot H^1$ control over the solution.}  Using the local theory, we can also obtain suitable estimates on any fixed finite time interval.  Indeed, given $T_0>0$, we may use the local theory and the $1d$ Sobolev embedding $H^1\subset L^\infty$ to obtain
\[
\sup_{t\in[0,T_0]} \{ \|u(t)\|_X + \|u(t)\|_S\} \lesssim_{T_0} \eps. 
\]
In what follows, we will fix $T_0=T_0(\gamma)>0$ large enough that $\Gamma(t)\geq \tfrac12\langle\gamma\rangle t$ for all $t\geq T_0$, where $\langle\gamma\rangle>0$ is the average dispersion.  That this is possible follows from the fact that $|\Gamma(t)-\langle \gamma \rangle t| \lesssim 1$ uniformly in $t$ (cf. \eqref{Gammalemma}--\eqref{Gammalemma2} above).

It is straightforward to see that control of the $S$-norm yields control over the $J_\Gamma$ and $\nabla$ components of the $X$-norm.  In particular, using the Duhamel formula, the commutation properties of $J_\Gamma(t)$ with $e^{i\Gamma(t)\Delta}$, and the chain rule \eqref{chain} for $J_\Gamma$ we have
\begin{align*}
\|J_\Gamma(t) u(t)\|_{L^2} & \lesssim \|xu_0\|_{L_x^2} + \int_0^t \| J_\Gamma(s)[|u|^2 u]\|_{L_x^2}\,ds \\
& \lesssim \eps + \int_0^t \langle s\rangle^{-1}\|u(s)\|_S^2 \|J_\Gamma(s)u(s)\|_{L_x^2}\,ds,
\end{align*}
which implies (via Gronwall's inequality) that
\begin{equation}\label{energy-type1}
\|J_\Gamma(t)u(t)\|_{L^2} \lesssim \eps\langle t\rangle^{C\|u(t)\|_S^2}
\end{equation}
for some absolute $C>0$ and all $t>0$.  We can repeat this argument with the gradient to obtain
\begin{equation}\label{energy-type2}
\|\nabla u(t)\|_{L^2} \lesssim \eps\langle t\rangle^{C\|u(t)\|_S^2}.
\end{equation}
Thus we find that (for $\eps$ sufficiently small), we may derive
\[
\|u(t)\|_S \lesssim \eps \implies \|u(t)\|_X \lesssim \eps, 
\]
which is the first half of the bootstrap argument.

To complete the bootstrap argument, we need to show the converse, namely, that control over the $X$-norm implies control over the $S$-norm.  To this end, we recall the notation defined after \eqref{factorization} and introduce a new variable $w$ on the interval $[T_0,\infty)$ via 
\begin{equation}\label{def-w}
u(t) = M(\Gamma(t))D(\Gamma(t))w(t). 
\end{equation}
By direct computation (using \eqref{dmnls}), it follows that $w$ solves the equation
\[
i\partial_t w +[2\Gamma(t)]^{-2}\gamma(t)\Delta w = - |2\Gamma(t)|^{-1}|w|^2 w. 
\]

We note also that (since $\Gamma(t)\sim t$)
\[
\|w(t)\|_{L^2}\equiv\|u(t)\|_{L^2} \qtq{and} \|w(t)\|_{L^\infty}\sim t^{\frac12}\|u(t)\|_{L^\infty}.
\]
Furthermore, using \eqref{J-identities}, we have
\[
\|\nabla w(t)\|_{L^2} = \|J_\Gamma(t) u(t)\|_{L^2}\lesssim t^{\delta}\|u(t)\|_X.
\]
We also observe that (as $0<\delta\ll 1$)
\[
\|xw(t)\|_{L^2} = \|\tfrac{x}{2\Gamma(t)}u(t)\|_{L^2} \lesssim |\Gamma(t)|^{-1} \|J_\Gamma(t) u(t)\|_{L^2} + \|\nabla u(t)\|_{L^2}\lesssim t^{\delta}\|u(t)\|_X. 
\]

Our primary goal is to estimate $w$ in $L_x^\infty$. Note that by the Sobolev embedding $H^1\subset L^\infty$, we know already that
\[
\|w(t)\|_{L^\infty} \lesssim \|w(t)\|_{H^1} \lesssim t^{\delta}\|u(t)\|_X. 
\]
Furthermore, we can obtain a suitable estimate for the high frequencies of $w$ as follows: by Bernstein estimates (see Lemma~\ref{L:Bernstein}), 
\begin{align*}
\|P_{>\sqrt{t}}w(t)\|_{L^\infty} & \lesssim \sum_{N>\sqrt{t}}\|P_N w(t)\|_{L^\infty} \\
& \lesssim \sum_{N>\sqrt{t}}N^{-\frac12}\|\nabla w(t)\|_{L_x^2} \\
& \lesssim \sum_{N>\sqrt{t}} N^{-\frac12}t^{\delta}\|u(t)\|_X \lesssim t^{-\frac14+\delta}\|u(t)\|_X,
\end{align*}
where the sum is restricted to $N\in 2^{\mathbb{Z}}$.  It therefore remains to estimate the low frequency component. 

To this end, let us define
\[
\tilde w := P_{\leq \sqrt{t}} w
\]
and consider the equation satisfied by $\tilde w$.  Noting that the frequency projection is time-dependent, we compute 
\begin{equation}\label{w-eqn}
\begin{aligned}
i\partial_t \tilde w +[2\Gamma(t)]^{-1}|\tilde w|^2 \tilde w & = -[2\Gamma(t)]^{-2}\gamma(t)\Delta \tilde w - \tfrac{1}{2}t^{-\frac{3}{2}}\tilde P_{\sqrt{t}}\nabla w \\
& \quad \quad -[2\Gamma(t)]^{-1}[P_{\leq\sqrt{t}}(|w|^2 w) - |\tilde w|^2 \tilde w], \\
\end{aligned}
\end{equation}
where $\tilde P_{\sqrt{t}}$ is the Fourier multiplier operator with symbol $\psi'(\tfrac{\cdot}{\sqrt{t}})$. Here we recall that $\psi$ is the multiplier appearing in the definition of  the Littlewood--Paley projection operators; in particular, $\psi'(\tfrac{\cdot}{\sqrt{t}})$ is supported near frequencies $\sqrt{t}$. 

We will be able to obtain integrable estimates for the terms on the right-hand side of \eqref{w-eqn}.  To deal with the non-integrable term on the left-hand side, we introduce the unimodular integrating factor
\[
B(t) = \exp\biggl\{-i\int_{T_0}^t |\tilde w(s)|^2\tfrac{ds}{2\Gamma(s)}\biggr\},\qtq{and set} g(t) = B(t)\tilde w(t). 
\]
Then $g$ solves
\begin{align*}
i\partial_t g & = B(t)\bigl\{ -[2\Gamma(t)]^{-2}\gamma(t)\Delta\tilde w - \tfrac12 t^{-\frac32}\tilde P_{\sqrt{t}} \nabla w \\
& \quad\quad\quad - [2\Gamma(t)]^{-1}[P_{\leq \sqrt{t}}(|w|^2 w) - |\tilde w|^2 \tilde w]\bigr\}
\end{align*}

We now estimate the right-hand side of this equation in $L_x^\infty$.

First, the linear terms are controlled using Bernstein estimates (Lemma~\ref{L:Bernstein}): 
\begin{align*}
t^{-2}\|\Delta \tilde w\|_{L_x^\infty} + t^{-\frac32}\|\tilde P_{\sqrt{t}}\nabla w\|_{L_x^\infty} \lesssim t^{-\frac54} \|\nabla w\|_{L_x^2} \lesssim t^{-\frac54+\delta}\|u\|_X. 
\end{align*}
For the nonlinear term, we begin by writing
\begin{align*}
\Gamma(t)^{-1}\| P_{\leq \sqrt{t}}(|w|^2 w)-|\tilde w|^2 \tilde w \|_{L_x^\infty} & \lesssim t^{-1}\|P_{>\sqrt{t}}(|w|^2 w)\|_{L_x^\infty} \\
& \quad + t^{-1} \| |w|^2 w - |\tilde w|^2 \tilde w\|_{L_x^\infty}. 
\end{align*}
Using Bernstein estimates as above, we first obtain
\begin{align*}
t^{-1} \|P_{>\sqrt{t}}(|w|^2 w)\|_{L_x^\infty}  \lesssim t^{-\frac54}\|\nabla(|w|^2 w)\|_{L_x^2} & \lesssim t^{-\frac54} \|w\|_{L^\infty}^2 \|\nabla w\|_{L_x^2} \\
&\lesssim t^{-\frac54+3\delta}\|u\|_X^3. 
\end{align*}
We next observe\footnote{The $\O$ notation means that we can write the expression on the left-hand side as a finite linear combination of terms of the form on the right, up to additional frequency projections and complex conjugation.} that 
\[
|w|^2 w - |\tilde w|^2 \tilde w = \O(w^2 P_{>\sqrt{t}}w),
\] 
so that (by the high frequency estimate above) 
\begin{align*}
\Gamma(t)^{-1}\| |w|^2 w - |\tilde w|^2 \tilde w\|_{L_x^\infty} \lesssim t^{-1}\|w\|_{L_x^\infty}^2 \|P_{>\sqrt{t}} w\|_{L_x^\infty} & \lesssim t^{-\frac54+3\delta}\|u\|_X^3. 
\end{align*}

Noting that $\|g\|_{L_x^\infty}\equiv \|w\|_{L_x^\infty}$ and recalling that $0<\delta\ll 1$, it follows that
\begin{equation}\label{dispersive-type}
\begin{aligned}
\|w(t)\|_{L_x^\infty} & \lesssim \|w(T_0)\|_{L_x^\infty} + \|P_{>\sqrt{t}}w(t)\|_{L_x^\infty} \\
& \quad + \int_{T_0}^t s^{-\frac54+\delta}\|u(s)\|_X + s^{-\frac54+3\delta}\|u(s)\|_X^3\,ds \\
& \lesssim \eps + \|u(t)\|_X + \|u(t)\|_X^3,
\end{aligned}
\end{equation}
which (recalling that $\|w(t)\|_{L^\infty}\sim t^{\frac12}\|u(t)\|_{L^\infty}$) is the second bootstrap estimate we need.  

In particular, for $\eps$ and $\delta$ sufficiently small, a continuity argument now implies that
\[
\|u(t)\|_S+\|u(t)\|_X \lesssim \eps \qtq{for all}t\geq 0,
\]
yielding the desired decay for the solution $u$. For the sake of completeness, we include the details of the continuity argument in Appendix~\ref{S:A}.

It remains to demonstrate that $u$ has the asymptotic behavior described in \eqref{asymptotic}. 

To begin, note that with the estimates in hand, we can see immediately that
\[
\|g\|_{L_{t,x}^\infty} \lesssim \eps,\quad \|\partial_t g(t)\|_{L_x^\infty} \lesssim \eps t^{-\frac54+3\delta},\qtq{and}\|w(t)-\tilde w(t)\|_{L_x^\infty} \lesssim \eps t^{-\frac14+\delta}. 
\]
Thus we may define $W_0=\lim_{t\to\infty} g(t)$ (with the limit taken in $L_x^\infty$) and derive that
\begin{equation}\label{wwB1}
\|w(t) - B(t)^{-1}W_0\|_{L_x^\infty} \lesssim \eps t^{-\frac14+3\delta} 
\end{equation}
for all large $t$. 

Now we wish to extract the leading order behavior of the phase in $B(t)^{-1}$.  Recalling $|\tilde w|=|g|$, let us define $\Psi(t)$ via
\begin{equation}\label{phase-main}
\int_{T_0}^t |g(s)|^2 \tfrac{ds}{2\Gamma(s)} = \tfrac{1}{2\langle\gamma\rangle}|g(t)|^2\log(\tfrac{t}{T_0})+\Psi(t). 
\end{equation}

We will show that $\Psi$ is Cauchy in $L^\infty$.  To this end, we write
\begin{align}
\Psi(t)-\Psi(s) &= \int_s^t\bigl[|g(\tau)|^2-|g(t)|^2\bigr]\tfrac{d\tau}{2\langle\gamma\rangle \tau}\label{PCauchy1} \\
& \quad + \int_s^t |g(\tau)|^2\bigl[\tfrac{1}{2\Gamma(\tau)}-\tfrac{1}{2\langle\gamma\rangle\tau}\bigr]\,d\tau \label{PCauchy2} \\
& \quad - \tfrac{1}{2\langle\gamma\rangle}\bigl[|g(t)|^2 -|g(s)|^2]\log(\tfrac{s}{T_0}). \label{PCauchy3}  
\end{align}
Using the bounds on $g$ and $\partial_t g$, we find that
\[
|\eqref{PCauchy1}| \lesssim \eps^2 s^{-\frac14+3\delta}.
\]
Next, using $|\Gamma(\tau)-\langle\gamma\rangle \tau|\lesssim 1$ (cf. \eqref{Gammalemma}), we find
\[
|\eqref{PCauchy2}| \lesssim \eps^2 s^{-1}.
\]
Finally,
\[
|\eqref{PCauchy3}| \lesssim \eps^2 s^{-\frac14+3\delta}\log s.
\]
It follows that $\Psi(t)$ is Cauchy and hence converges to some $\Phi$ in $L^\infty$ as $t\to\infty$, with
\[
\|\Psi(t)-\Phi\|_{L_x^\infty} \lesssim \eps^2 t^{-\frac14+4\delta}.
\]

Consequently, we obtain from \eqref{wwB1} and \eqref{phase-main} that 
\[
\|w(t) - \exp\{\tfrac{i}{2\langle\gamma\rangle}|g(t)|^2 \log(\tfrac{t}{T_0})+i\Phi\}W_0\|_{L_x^\infty} \lesssim \eps t^{-\frac14+4\delta}.
\]
Defining
\[
W= \exp\{i[\Phi - \tfrac{1}{2\langle\gamma\rangle}|W_0|^2\log(T_0)]\}W_0,
\]
we derive
\[
\|w(t) - \exp\{\tfrac{i}{2\langle\gamma\rangle}|W|^2\log t\}W\|_{L_x^\infty} \lesssim t^{-\frac14+4\delta}. 
\]

Recalling the definition of $w(t)$ in \eqref{def-w}, we obtain the asymptotic formula
\[
u(t,x) = (2i\Gamma(t))^{-\frac12}\exp\bigl\{\tfrac{ix^2}{4\Gamma(t)}+\tfrac{i}{2\langle\gamma\rangle}|W(\tfrac{x}{2\Gamma(t)})|^2\log t\bigr\}W(\tfrac{x}{2\Gamma(t)})+\mathcal{O}(t^{-\frac34+4\delta})
\]
in $L_x^\infty$ as $t\to\infty$, which is a quantitative form of the asymptotic formula \eqref{asymptotic} appearing in Theorem~\ref{T1}.  Thus the proof of Theorem~\ref{T1} is complete.

\appendix

\section{Continuity argument}\label{S:A}

This section contains the details of the continuity argument used in Section~\ref{S:main}; in particular, we import all notation from that section.

Given $\eps>0$, we firstly obtain that $\|u(t)\|_S,\|u(t)\|_X\leq C_0\eps$ on $[0,T_0]$ for some universal $C_0=C_0(\gamma)>0$.  Thus it suffices to establish bounds on $[T_0,\infty)$.  We let
\[
I=\{t\geq T_0: \|u(t)\|_S \leq C_1 \eps \qtq{and}\|u(t)\|_X\leq C_2\eps\},
\]
and we will show that for $C_1,C_2$ large enough and $\eps>0$ small enough (depending on universal constants), we have $I=[T_0,\infty)$. Noting that $I\ni T_0$ for $C_1,C_2\geq C_0$ and that $I$ is closed by continuity of the flow, it suffices to prove that $I$ is open.

To this end, suppose $t\in I$. It follows from the energy estimates \eqref{energy-type1}--\eqref{energy-type2} that
\[
\|u(t)\|_{L^2}+ \|\nabla u(t)\|_{L^2} + \|J_\Gamma(t)u(t)\|_{L^2} \leq \tilde C_1\eps\langle t\rangle^{cC_1^2 \eps^2}
\]
for some universal $\tilde C_1, c>0$.  This implies
\[
\|u(t)\|_X \leq \tfrac12 C_2\eps
\]
provided we can arrange 
\begin{equation}\label{continuity1}
cC_1^2 \eps^2 <\delta \qtq{and} C_2\geq 2\tilde C_1
\end{equation}
(recall that $\delta$ is a small but fixed parameter, e.g. $\delta=\tfrac{1}{100}$ would suffice). 

On the other hand, it follows from the dispersive-type estimate \eqref{dispersive-type} that 
\[
\|u(t)\|_S\leq \tilde C_2 [(C_0+C_2)\eps + (C_2\eps)^3]
\]
for some universal $\tilde C_2>0$. This implies
\[
\|u(t)\|_S \leq \tfrac12 C_1 \eps
\]
provided we can arrange 
\begin{equation}\label{continuity2}
C_1\geq 2\tilde C_2[C_0+C_2+C_2^3\eps^2]. 
\end{equation}

In particular, we can arrange both \eqref{continuity1} and \eqref{continuity2} by first choosing $C_2=\max\{C_0,2\tilde C_1\}$ and $C_1=2\tilde C_2[C_0+C_2+C_2^3]$, and ensuring that $\eps=\eps(c,C_0,\delta)$ is sufficiently small.

Applying continuity of the flow once again, we find that $(t-\eta,t+\eta)\subset I$ for some $\eta>0$.  Thus $I$ is open as well, and hence $I=[T_0,\infty)$.


\end{document}